\documentclass{sig-alternate}
\usepackage[english]{babel}
\usepackage[utf8]{inputenc}
\usepackage{color}
\usepackage{ifthen}
\usepackage{stmaryrd}
\usepackage[
  pdftex,
  bookmarks=true,
  bookmarksnumbered=true,
  hypertexnames=false,
  breaklinks=true,
  linkbordercolor={0 0 1}
  ]{hyperref}
\usepackage[ruled,lined,vlined,linesnumbered,titlenumbered,noend]{algorithm2e}

\hypersetup{
  pdfauthor = {Amaury Pouly},
  pdftitle = {On the complexity of solving initial value problems},
  pdfsubject = {},
  pdfkeywords = {},
  pdfcreator = {pdfLaTeX, with hyperref},
  pdfproducer = {pdfLaTeX},
}

\allowdisplaybreaks

\newcommand{\R}{\mathbb{R}}
\newcommand{\Q}{\mathbb{Q}}
\newcommand{\N}{\mathbb{N}}

\newdef{definition}{Definition}
\newtheorem{lemma}{Lemma}
\newtheorem{theorem}{Theorem}
\newtheorem{proposition}{Proposition}
\newtheorem{corollary}{Corollary}
\newtheorem{hypothesis}{Hypothesis}

\newcommand{\lemref}[1]{Lemma~\ref{#1}}

\newcommand{\propref}[1]{Proposition~\ref{#1}}
\newcommand{\corref}[1]{Corollary~\ref{#1}}

\newcommand{\secref}[1]{Section~\ref{#1}}
\newcommand{\algref}[1]{Algorithm~\ref{#1}}
\newcommand{\hypref}[1]{Hypothesis~\ref{#1}}

\newcommand{\infnorm}[1]{\left\lVert{#1}\right\rVert_\infty}
\newcommand{\sigmap}[1]{{\Sigma{#1}}}
\newcommand{\degp}[1]{\operatorname{deg}(#1)}
\newcommand{\multifac}[2]{{#2}!^{(#1)}}
\newcommand{\fiter}[2]{{#1}^{[#2]}}

\newcommand{\myforall}{\forall\thinspace}
\newcommand{\fnsup}[2]{{S_{#1}#2}}
\newcommand{\taylor}[3]{{T_{#1}^{#2}#3}}

\begin{document}

\title{On the complexity of solving initial value problems}
\numberofauthors{3}
\author{
   \alignauthor
  Olivier Bournez \\
\affaddr{Ecole Polytechnique, LIX}\\
 \affaddr{91128 Palaiseau Cedex, France. } \\
    \email{olivier.bournez@lix.polytechnique.fr}
\and
Daniel S. Gra\c{c}a\\
\affaddr{CEDMES/FCT, Universidade do Algarve, C. Gambelas, 8005-139 Faro,
  Portugal.} \\
\affaddr{SQIG /Instituto de Telecomunica\c{c}\~{o}es, Lisbon,
  Portugal.}\\
\email{dgraca@ualg.pt}
    \and
    \alignauthor Amaury Pouly\\
\affaddr{Ecole Normale Sup\'{e}rieure de Lyon, France.} \\
       \email{amaury.pouly@ens-lyon.fr}
}
\date{\today}
\maketitle

\begin{abstract}
  In this paper we prove that computing the solution of an
  initial-value problem $\dot{y}=p(y)$ with initial
  condition $y(t_0)=y_0\in\R^d$ at time $t_0+T$ with precision
  $e^{-\mu}$ where $p$ is a vector of polynomials can be done in time
  polynomial in the value of $T$, $\mu$ and $Y=\sup_{t_0\leqslant
    u\leqslant T}\infnorm{y(u)}$. Contrary to existing results, our
  algorithm works for any vector of polynomials $p$ over any bounded or
  unbounded domain and has a guaranteed complexity and precision. In
  particular we do not assume $p$ to be fixed, nor the solution to lie
  in a compact domain, nor we assume that $p$ has a Lipschitz
  constant.
\end{abstract}

\section{Introduction}

Solving initial-value problems (IVPs) defined with ordinary
differential equations (ODEs) is of great interest, both in practice
and in theory. Many algorithms have been devised to solve IVPs, but
they usually only satisfy one of the following conditions: (i) they
are guaranteed to find a solution of the IVP with a given precision
(accuracy) (ii) they are fast (efficiency). It is not easy to find an
algorithm which satisfies both (i) and (ii) since, in practice,
efficiency comes at the cost of accuracy and vice versa.

Observe in particular that when dealing with functions defined on an
unbounded domain (say $\R$, the set of reals),  numerical methods with guaranteed
accuracy for
solving ODE are not polynomial\footnote{In the classical sense, that is to say
to be formal, in the sense of recursive analysis \cite{Wei00}.} even
for very basic functions.  Indeed, usual methods for numerical
integrations (including basic Euler's method, Runge Kutta's methods,
etc) fall in the general theory of  $n$-order methods for some
$n$. They do require a polynomial number of steps for general functions over a
compact domain $[a,b]$, but are not always polynomial over unbounded
domains: computing $f(t)$, hence solving the ODE over a domain of the
form $[0,t]$ is not done in a number of steps polynomial in $t$
without further hypotheses on function $f$ by any $n$-order method
(see e.g. the general theory in \cite{Dem96}).
This has been already observed in \cite{Smi1}, and ODEs have been
claimed to be solvable in polynomial time in \cite{Smi1} for some classes
of functions by using methods of order $n$ with $n$ depending on $t$,
but without a full proof.

Somehow, part of the
problem is that solving $\dot{y}=f(y)$ in full generality requires
some knowledge of $f$. Generally speaking, most algorithms only work
over a fixed, specified compact domain, since in this case $f$ is
usually Lipschitz (it is well-known that any $C^1$ function over a
compact is also Lipschitz there).

In this paper we present an algorithm which has been designed to work
over an unbounded domain, be it time or space. More precisely, we do
not assume that the solution lies in a compact domain. Our algorithm
guarantees the precision (given as input) and its running time is
analyzed. Achieving both aims is especially complicated over unbounded
domains, because we do not have access to the classic and fundamental
assumption that $f$ is Lipschitz.

The proposed method is
based on the idea of using a varying order: the order $n$ is chosen
accordingly to the required precision and other parameters.  Compared to
\cite{Smi1} (which takes a physicist's point of view, being more interested on the physics of the
$n$-body problem than on studying the problem from a numerical analysis or recursive analysis perspective), we provide full proofs, and we state
precisely the required hypotheses.

Even though many useful functions $f$ are locally Lipschitz functions,
there is an inherent chicken-and-egg problem in using this hypothesis.
We can always compute some approximation $z(t)$ of the correct
solution $y(t)$. However, to know the error we made computing $z(T)$,
one often needs a local Lipschitz constant $L$ which is valid for some set $A$
with the property that $y(t),z(t)\in A$ for all $t\in [t_0,T]$.
However, we can only obtain the Lipschitz constant if we know $A$, but
we can only know $A$ from $z(t)$ if we know the error we made
computing $z(t)$, which was the problem we were trying to solve in the
first place.

In this paper we are interested in obtaining an algorithm which provides
an approximation of the solution of $\dot{y}=f(y)$, bounded by error
$\varepsilon$, where $\varepsilon$ is given as input (as well as other
parameters). We will then analyze the impact of the various parameters
on the runtime of the algorithm.

We mainly restrict in this paper our analysis to the simple yet broad
class of differential equations of the form \eqref{eq:ode}, that is
$\dot{y}=p(y)$ where $p$ is a vector of polynomials. This is motivated
first by the fact that most ODEs using usual functions from Analysis
can be rewritten equivalently into the format \eqref{eq:ode} -- see
\cite{GCB08}. We denote the solutions of this kind of problems as
PIVP functions.

The necessary input parameters for the
algorithm will be specified as follows: Given $p$ and a solution $y:I\rightarrow\R^d$ to the PIVP
\eqref{eq:ode}, we want to compute $y(T)$ with a precision $2^{-\mu}$
knowing that $\forall t\leqslant T, \infnorm{y(T)}\leqslant Y$. Our
parameters are thus $d$ the dimension of the system, $T$ the time at
which to compute, $\mu$ the precision, $Y$ a bound on the solution and
$k=\degp{p}$ the degree of the polynomials.

Our second motivation for studying the particular class \eqref{eq:ode} of
polynomial IVPs comes from the study of the General Purpose Analog
Computer (GPAC) \cite{Sha41}. The GPAC is an analog model of
computation introduced by Claude Shannon as an idealization of an
analog computer, the Differential Analyzer, which most well-known implementation was done
 in 1931 by Vannevar Bush \cite{Bus31}. Differential Analyzers have been
used intensively up to the 1950's as computational machines to solve various
problems from ballistic to aircraft design, before the era of digital
computations that was boosted by the invention of the transistor
\cite{IEEEAnnals}.

It is
known that any GPAC can be described equivalently as the solution of a
PIVP \cite{GC03}. It has been shown that the
GPAC (and thus PIVP functions) are equivalent to Turing machines from
a computability point of view \cite{GCB08}, \cite{BCGH07}. However, it is
unknown whether this equivalence holds at a complexity level. This
work is a required and substantial step towards comparing the GPAC to
Turing machines, or if one prefers in more provocative terms, in
proving the not clear fact that  analog computation is not stronger than digital
computation. See \cite{CIEChapter2007} and \cite{JOC2007} for more discussions.

In other words, with the results presented in this paper, we seek not only to
understand what is the computational complexity of the problem of
solving PIVPs (which appear in many applications from Physics, etc.),
but also to understand how analog and digital computational models can
be related at a computational complexity level.

\noindent\textbf{Organization of the paper}\\
In \secref{sec:not_basics} we introduce some notations and claim some
basic results that will be useful later. In \secref{sec:nth_deriv} we
derive an explicit bound on the derivatives of the solution at any
order. In \secref{sec:depend_init_cond} we derive an explicit bound on
the divergence of two solutions given the initial difference. In
\secref{sec:taylor_approx} we apply the results of the previous
section to derive an explicit error bound for a Taylor approximation
of the solution at any point. In \secref{sec:result} we describe an
algorithm to solve a PIVP and give an explicit complexity. We hence
obtain the proof of our main result.
\medskip

\noindent\textbf{Overview of the paper}\\
In order to compute the solution to a PIVP, we use a classical
multi-step method, but of varying order. At each step, we use a Taylor
approximation of the solution at an arbitrary order to approximate the
solution. In \secref{sec:nth_deriv} we explain how to compute the
derivatives needed by the Taylor approximation. Furthermore, we need
an explicit bound on the derivatives since our method is not of fixed
order. \secref{sec:nth_deriv} provides such a bound as a
corollary. Since our method computes an approximation of the solution at
each point, it will make slight errors that might amplify if not dealt
with correctly. We can control the errors in two ways: by reducing the
time step and by increasing the order of the method. A careful balance
between those factors is needed. In \secref{sec:depend_init_cond} we
explain how the error grows when the initial condition of the PIVP is
perturbed. In \secref{sec:taylor_approx} we quantify the overall error growth,
by also including the error introduced by using
Taylor approximations. Finally in \secref{sec:result} we put
everything together and explain how to balance the different
parameters to get our main result.


\section{Notation and basic facts}\label{sec:not_basics}

We will use the following notations:
\[\infnorm{(x_1,\ldots,x_n)}=\max_{1\leqslant i\leqslant n}|x_i|\]
\[\lVert(x_1,\ldots,x_n)\rVert=\sqrt{|x_1|^2+\cdots+|x_n|^2}\]
\[x\leqslant y\Leftrightarrow \myforall i,\thinspace x_i\leqslant y_i\]
\[\multifac{d}{n}=d^nn!\]
\[\fiter{f}{n}=
\begin{cases}
\operatorname{id} & \text{if }n=0\\
\fiter{f}{n-1}\circ f & \text{otherwise}
\end{cases}\]
\[\fnsup{a}{f}(t)=\sup_{a\leqslant u\leqslant t}\infnorm{f(u)}\]
\[\taylor{a}{n}{f}(t)=\sum_{k=0}^{n-1}\frac{f^{(k)}(a)}{k!}(t-a)^k\]

\noindent Note that $\leqslant$ is not an order on $\R^d$ but just a notation. We will denote by $W$ the principal branch of the Lambert W function which satisfies $x=W(x)e^{W(x)}$ and $W\geqslant-1$.

\noindent We will consider the following ODE:
\begin{equation}\left\{\begin{array}{@{}c@{}l}\dot{y}&=p(y)\\y(t_0)&=y_0\end{array}\right.\label{eq:ode}\end{equation}
where $p:\R^d\rightarrow\R^d$ is a vector of polynomial. If $p:\R^d\rightarrow\R$ is polynomial, we write:
\[p(X_1,\ldots,X_d)=\sum_{|\alpha|\leqslant k}a_\alpha X^\alpha\]
where $k$ is degree of $p_i$ that will be written $\degp{p_i}$; We write $|\alpha|=\alpha_1+\cdots+\alpha_d$. We will also write:
\[\sigmap{P}=\sum_{|\alpha|\leqslant k}|a_\alpha|\]
If $p:\R^d\rightarrow\R^d$ is a vector of polynomial, we write $\degp{p}=\max(\degp{p_1},\ldots,\degp{p_d})$ and $\sigmap{p}=\max(\sigmap{p_1},\ldots,\sigmap{p_d})$.
With the previous notation, the following lemmas are obvious.
\begin{lemma}\label{lem:poly_major_sigmap_deg}
For any $Q:\R^n\rightarrow\R$ and any $x\in\R^n$,
\[|Q(x)|\leqslant\sigmap{Q}\max(1,\infnorm{x}^{\degp{Q}})\]
\end{lemma}
\begin{lemma}\label{lem:poly_deriv_major_sigmap_deg}
For any polynomial $Q:\R^n\rightarrow\R$, $\alpha\in\N^d$ and $x\in\R^n$, if $|\alpha|\leqslant\degp{Q}$ then
\[|Q^{(\alpha)}(x)|\leqslant|\alpha|!\sigmap{Q}\max(1,\infnorm{x}^{\degp{Q}-|\alpha|})\]
\end{lemma}

\section{Nth derivative of $y$}\label{sec:nth_deriv}

Given the relationship between $\dot{y}$ and $p(y)$ it is natural to try to extend it to compute the $n$th derivative of $y$. Of particular interest is the following remark: the derivatives at a point $t$ only depend on $y(t)$. Since an exact formula is difficult to obtain, we only give a recursive formula and try to bound the coefficients of the relationship obtained.

Notice that since $p_i$ is a polynomial, it is infinitely differentiable and thus the partial derivatives commute. This means that for any $\alpha\in\N^d$, $p_i^{(\alpha)}$ is well-defined and is independent of the order in which the derivatives are taken.

\begin{proposition}\label{prop:nth_derivative}
If $y$ satisfies \eqref{eq:ode} for any $t\in I$ and $\degp{p}=k$, define:
\[\Gamma=\llbracket0,k\rrbracket^d\qquad\Lambda=\llbracket1,d\rrbracket\times\Gamma\]
\[V(t)=\left(p_i^{(\alpha)}(y(t))\right)_{(i,\alpha)\in\Lambda}\qquad t\in I\]
Then
\[\forall t\in I,\forall n\in\N^{*}, y_i^{(n)}(t)=Q_{i,n}(V(t))\]
where $Q_{i,n}$ is a polynomial of degree at most $n$. Furthermore,
\[\sigmap{Q_{i,n}}\leqslant\multifac{d}{(n-1)}\]
\end{proposition}

\begin{proof}
First notice that $Q_{i,n}$ has variables $p_j^{(\alpha)}(y(t))$ so we will use $\partial_{j,\alpha}Q_{i,n}$ to designate its partial derivatives. We will prove this result by induction on $n$. The case of $n=1$ is trivial:
\[y_i'(t)=p_i(y(t))\qquad Q_{i,1}(V(t))=p_i(y(t))\]
\[\degp{Q_{i,1}}=1\qquad\sigmap{Q_{i,1}}=1=\multifac{d}{0}\]
Now fix $n\geqslant1$. We will need $\mu_k\in\N^d$ such that $(\mu_k)_i=\delta_{i,k}$. Elementary differential calculus gives:
\begin{align*}
y_i^{(n+1)}(t)&=\frac{dy_i^{(n)}}{dt}=\frac{d}{dt}\Big(Q_{i,n}(V(t))\Big)\\
    &=\sum_{j=1}^d\sum_{\alpha\in\Gamma}\frac{d}{dt}\Big(p_j^{(\alpha)}(y(t))\Big)\partial_{j,\alpha}Q_{i,n}(V(t))\\
    &=\sum_{j=1}^d\sum_{\alpha\in\Gamma}\left(\sum_{k=1}^d\dot{y}_k(t)\partial_kp_j^{(\alpha)}(y(t))\right)\partial_{j,\alpha}Q_{i,n}(V(t))\\
    &\hspace{-1.5em}=\sum_{j=1}^d\sum_{\alpha\in\Gamma}\left(\sum_{k=1}^dp_k(y(t))p_j^{(\alpha+\mu_k)}(y(t))\right)\partial_{j,\alpha}Q_{i,n}(V(t))
\end{align*}
Since $Q_{i,n}$ is a polynomial, this proves that $Q_{i,n+1}$ is a polynomial. Furthermore a close look at the expression above makes it clear that each monomial has degree at most $n+1$ since every monomial of $\partial_{j,\alpha}Q_{i,n}(V(t))$ has degree $n-1$ and is multiplied by the product of two variables of degree $1$.

Now, we can bound the sum of the coefficients. We will first need to give an explicit expression to $Q_{i,n}$ so we write:
\[Q_{i,n}=\sum_{|\beta|\leqslant n}a_\beta X^\beta\]
Recall that the variables of $Q_{i,n}$ are $p_j^{(\alpha)}$ so $\beta\in\N^{\llbracket1,d\rrbracket\times\Gamma}$ and $\beta_{i,\alpha}$ makes perfect sense. Then:
\begin{align*}
\sigmap{Q_{i,n+1}}&\leqslant\sum_{j=1}^d\sum_{\alpha\in\Gamma}\left(\sum_{k=1}^d1\right)\sigmap{\partial_{j,\alpha}Q_{i,n}}\\
    &=\sum_{j=1}^d\sum_{\alpha\in\Gamma}d\sigmap{\partial_{j,\alpha}\left(\sum_{|\beta|\leqslant n}|a_\beta| X^\beta\right)}\\
    &=\sum_{j=1}^d\sum_{\alpha\in\Gamma}d\sum_{|\beta|\leqslant n}|a_\beta|\beta_{i,\alpha}\\
    &=d\sum_{|\beta|\leqslant n}|a_\beta|\sum_{j=1}^d\sum_{\alpha\in\Gamma}\beta_{i,\alpha}\\
    &=d\sum_{|\beta|\leqslant n}|a_\beta|\sum_{j=1}^d|\beta|\\
    &=dn\sigmap{Q_{i,n}}\\
    &\leqslant dn\multifac{d}{(n-1)}\\
    &=\multifac{d}{n}
\end{align*}
\end{proof}




\begin{corollary}\label{cor:nth_derivative_simpl}
If $y$ satisfies \eqref{eq:ode} for $t\in I$ and $\degp{p}=k$, then
\[\infnorm{y^{(n)}(t)}\leqslant\multifac{d}{n}\max\left(1,k!\sigmap{p}\max\left(1,\infnorm{y(t)}^k\right)\right)^n\]
\end{corollary}

\section{Dependency in the initial condition}\label{sec:depend_init_cond}

When using multi-steps methods to solve ODEs or simply when doing approximation, we might end up solving an ODE like \eqref{eq:ode} with a wrong initial condition. This will of course affect the result of the computation since even if we could compute the solution with an infinite precision, the results would be different because of the dependency in the initial condition. For this reason we would like to evaluate this dependency numerically. Assuming $y$ satisfies \eqref{eq:ode}, we define the functional $\Phi$ as:
\[\Phi(t_0,y_0,t)=y(t)\]
Notice that the dependency in $y_0$ is implicit but that in particular, $\Phi(t_0,y_0,t_0)=y_0$. Also notice that $\Phi$ and $y_0$ are vectors so we'll study the dependency of $\Phi_i$ in $y_{0j}$.

We first recall the well-known Gronwall's inequality.

\begin{proposition}[Generalized Gronwall's inequality]\label{prop:gronwall}
Suppose $\psi$ satisfies
\[\psi(t)\leqslant\alpha(t)+\int_0^t\beta(s)\psi(s)ds,\qquad t\in[0,T]\]
with $\alpha(t)\in\R$ and $\beta(s)\geqslant0$. Then $\forall  t\in[0,T]$,
\[\psi(t)\leqslant\alpha(t)+\int_0^t\alpha(s)\beta(s)\exp\left(\int_s^t\beta(u)du\right)ds\]
If, in addition, $\alpha$ is a non-decreasing function on $[0,T]$, then
\[\psi(t)\leqslant\alpha(t)\exp\left(\int_0^t\beta(s)ds\right),\qquad t\in[0,T]\]
\end{proposition}

In order to apply Gronwall's inequality to $\phi$, we will need to bound the Lipschitz constant for a multivariate polynomial. So consider a polynomial $P\in\R[X_1,\ldots,X_d]$ and write:
\[P=\sum_{|\alpha|\leqslant k}a_\alpha X^\alpha\]

We first prove a lemma on monomials and then extend it to polynomials.
\begin{lemma}
If $a,b\in\R^d$, $\alpha\in\N^d$ and $\infnorm{a},\infnorm{b}\leqslant M$ then:
$$|b^\alpha-a^\alpha|\leqslant|\alpha|M^{|\alpha|-1}\infnorm{b-a}$$
\end{lemma}

\begin{proof}
One can see by induction that:
$$b^\alpha-a^\alpha=\sum_{i=1}^d\Big(\prod_{j<i}b_j^{\alpha_j}\Big)(b_i^{\alpha_i}-a_i^{\alpha_i})\Big(\prod_{j>i}a_j^{\alpha_j}\Big)$$
Since it is well know that for any integer $n$:
$$b^n-a^n=(b-a)\sum_{i=0}^{n-1}a^ib^{n-1-i}$$
Thus we can deduce that:
\begin{align*}
|b^\alpha-a^\alpha|&\leqslant\sum_{i=1}^d\Big(\prod_{j<i}|b_j|^{\alpha_j}\Big)|b_i^{\alpha_i}-a_i^{\alpha_i}|\Big(\prod_{j>i}|a_j|^{\alpha_j}\Big)\\
  &\leqslant\sum_{i=1}^dM^{|\alpha|-\alpha_i}|b-a|\sum_{j=0}^{\alpha_i-1}M^{\alpha_i-1}\\
  &\leqslant\infnorm{b-a}\sum_{i=1}^dM^{|\alpha|-1}\alpha_i\\
  &\leqslant|\alpha|\infnorm{b-a}M^{|\alpha|-1}
\end{align*}
\end{proof}

We can use this result to obtain an explicit Lipschitz bound for the polynomial $P$:
\begin{lemma}\label{lem:multivariate_poly_lipschitz}
For all $a,b\in\R^d$ such that $\infnorm{a},\infnorm{b}\leqslant M$,
$$|P(b)-P(a)|\leqslant kM^{k-1}\sigmap{P}\infnorm{b-a}$$
where $k=\deg P$.
\end{lemma}

\begin{proof}
\begin{align*}
|P(b)-P(a)|&\leqslant\sum_{|\alpha|\leqslant k}|a_\alpha||b^\alpha-a^\alpha|\\
  &\leqslant\sum_{|\alpha|\leqslant k}|a_\alpha||\alpha|M^{|\alpha|-1}\infnorm{b-a}\\
  &\leqslant kM^{k-1}\infnorm{b-a}\sum_{|\alpha|\leqslant k}|a_\alpha|\\
  &\leqslant kM^{k-1}\sigmap{P}\infnorm{b-a}
\end{align*}
\end{proof}

In order to evaluate the divergence between two solutions, we will need to solve a highly nonlinear equation (the so-called chicken-and-egg problem of the introduction). This lemma gives an explicit solution as well as an inequality result.

\begin{lemma}\label{lem:solve_W_alpha_beta_pow_k}
Let $\alpha,\beta,x\geqslant0$ and $k\in\N^{*}$ then:
\begin{equation}\label{eq:alpha_beta_x_k_W}
    x=\alpha e^{\beta x^{k}}\Leftrightarrow x=\alpha e^{-\frac{1}{k}W\left(-k\beta\alpha^{k}\right)}
\end{equation}
Furthermore,
\[k\beta\alpha^{k}\leqslant\frac{1}{3}\quad\Rightarrow\quad\text{\eqref{eq:alpha_beta_x_k_W} has a solution and }x\leqslant 4\alpha\]
\end{lemma}

\begin{proof}
If $\alpha=0$ this is trivial ($x=0$), so we can assume that $\alpha>0$ and write $\alpha=e^{\bar{\alpha}}$ and then $x>0$ so we can write $x=e^{\bar{x}}$. Then
\begin{align*}
    x=\alpha e^{\beta x^{k}}&\Leftrightarrow\bar{x}=\bar{\alpha}+\beta e^{k\bar{x}}\\
    &\Leftrightarrow \bar{g}=\beta e^{k\bar{\alpha}} e^{k\bar{g}}\qquad\text{where }\bar{g}=\bar{x}-\bar{\alpha}\\
    &\Leftrightarrow k\bar{g}=\bar{\beta}  e^{k\bar{g}}\qquad\text{where }\bar{\beta}=k\beta e^{k\bar{\alpha}}=k\beta\alpha^{k}\\
    &\Leftrightarrow \bar{h}=\bar{\beta}e^{\bar{h}}\qquad\text{where }\bar{h}=k\bar{g}\\
    &\Leftrightarrow -\bar{h}e^{-\bar{h}}=-\bar{\beta}\\
    &\Leftrightarrow \bar{h}=-W(-\bar{\beta})\\
    &\Leftrightarrow k(\bar{x}-\bar{\alpha})=-W(-k\beta \alpha^{k})\\
    &\Leftrightarrow \bar{x}=\bar{\alpha}-\frac{1}{k}W(-k\beta \alpha^{k})\\
    &\Leftrightarrow x=\alpha e^{-\frac{1}{k}W(-k\beta \alpha^{k})}\\
\end{align*}
\end{proof}

And finally we can apply this result to $\Phi$.

\begin{proposition}\label{prop:dependency_init_cond}
Let $I=[a,b]$ and $y_0,z_0\in\R^d$. Assume that $y=\Phi(a,y_0,\cdot)$ and $z=\Phi(a,z_0,\cdot)$ are defined over $I$. Let $Y=\fnsup{a}{y}$.
Assume that $\forall t\in I$,
\begin{equation}\label{eq:dependency_init_cond_hyp}
\infnorm{y_0-z_0}\exp\left(k^22^k\sigmap{p}|t-a|(1+Y(t)^{k-1})\right)\leqslant\frac{1}{3}
\end{equation}
Then $\forall t\in I$,
\[\infnorm{z(t)-y(t)}\leqslant\infnorm{z_0-y_0}e^{k(2+Y(t))^{k-1}\sigmap{p}|t-a|}\]
where $k=\degp{p}$.
\end{proposition}

\begin{proof}
Define
$M=\max(Y,\fnsup{a}{z})$ and
consider $\psi(t)=\infnorm{z(t)-y(t)}$. By definition of $\Phi$ we have:
\[y(t)=y_0+\int_a^tp(y(u))du,\qquad t\in I\]
\[z(t)=z_0+\int_a^tp(z(u))du,\qquad t\in I\]
Applying a few inequalities and \lemref{lem:multivariate_poly_lipschitz}, we get
\begin{align*}
\psi(t)&\leqslant\infnorm{z_0-y_0}+\int_a^t\infnorm{p(z(u))-p(y(u))}du\\
    &=\underbrace{\infnorm{z_0-y_0}}_{\alpha(t)}+\int_a^t \underbrace{kM(t)^{k-1}\sigmap{p}}_{\beta(u)}\psi(u)du
\end{align*}
Finally, apply \propref{prop:gronwall} with $\alpha$ and $\beta$ being non-decreasing functions to get
\[\psi(t)\leqslant\psi(a)\exp\left(kM(t)^{k-1}\sigmap{p}|t-a|\right),\qquad t\in I\]
This inequality looks good except that there is an hidden dependency in $M$: $M$ depends on $\Phi(a,y_0,\cdot)$ and $\Phi(a,y_0,\cdot)$ and we seek one on $\Phi(a,y_0,\cdot)$ only. Since $\psi$ is the difference between the two solutions, we have the following bound on $M$:
\begin{equation}\label{eq:f_dev_W_psi}
M(t)\leqslant Y(t)+\underbrace{\fnsup{a}{\psi}(t)}_{=\Psi(t)},\quad t\in I
\end{equation}
Thus $\forall t\in I$,
\[\psi(t)\leqslant\underbrace{\psi(a)\exp\left(k\left(Y(t)+\Psi(t)\right)^{k-1}\sigmap{p}|t-a|\right)}_{G(Y(t),\Psi(t))}\]
And since $Y$, $\Psi$ and $G$ are non-decreasing functions, $G(Y(t),\Psi(t))$ is a non-decreasing function so we have:
\[\Psi(t)\leqslant G(Y(t),\Psi(t)),\qquad t\in I\]
Consider the solution $f:J\rightarrow\R$, $J\subseteq I$\footnote{Notice that J can't be empty because $a\in J$ since $f(a)=\psi(a)$. We will see that $J=I$} to:
\begin{equation}\label{eq:f_dev_W_f}
f(t)=\underbrace{\psi(a)\exp\left(k2^k\left(Y(t)^{k-1}+f(t)^{k-1}\right)\sigmap{p}|t-a|\right)}_{H(Y(t),f(t))}
\end{equation}
Since \eqref{eq:f_dev_W_f} implies $f(a)=\psi(a)=\Psi(a)$ and $0\leqslant x\leqslant y\Rightarrow\forall z\geqslant0,G(z,x)\leqslant H(z,y)$, then $\forall t\in J,\Psi(t)\leqslant f(t)$ and we can find an explicit expression for $f$:
\begin{align*}
\eqref{eq:f_dev_W_f}&\Leftrightarrow f(t)=\underbrace{\psi(a)\exp\left(k2^kY(t)^{k-1}\sigmap{p}|t-a|\right)}_{\alpha(t)}\\
    &\qquad\times\exp\big(\underbrace{k2^k\sigmap{p}|t-a|}_{\beta(t)}f(t)^{k-1}\big)\\
&\Leftrightarrow f(t)=\alpha(t)\exp\big(\beta(t)f(t)^{k-1}\big)\\
\end{align*}
Applying \lemref{lem:solve_W_alpha_beta_pow_k} and since \eqref{eq:dependency_init_cond_hyp} implies that $(k-1)\beta(t)\alpha(t)^{k-1}\leqslant\frac{1}{3}$ we have:
\[f(t)\leqslant 4\alpha(t),\qquad t\in I\]
Notice that the case of $k=1$ is not handled by \lemref{lem:solve_W_alpha_beta_pow_k} but trivially gives the same result.
Thus $\forall t\in I$,
\begin{align*}
M(t)&\leqslant Y(t)+4\psi(a)\exp\left(k2^kY(t)^{k-1}\sigmap{p}|t-a|\right)\\
    &\leqslant 2+Y(t)
\end{align*}
So finally,
\[\psi(t)\leqslant\psi(a)\exp\left(k(2+Y(t))^{k-1}\sigmap{p}|t-a|\right),\qquad t\in I\]
\end{proof}

\begin{corollary}\label{cor:dependency_init_cond_simp}
Let $I=[a,b]$ and $y_0,z_0\in\R^d$. Assume that $y=\Phi(a,y_0,\cdot)$ and $z=\Phi(a,z_0,\cdot)$ are defined over $I$. Let
$Y=\fnsup{a}{y}$ and $\mu\leqslant\frac{1}{3}$ and
assume that $\forall t\in I$,
\begin{equation}
\infnorm{y_0-z_0}\exp\left(k4^k\sigmap{p}|t-a|(1+Y(t)^{k-1})\right)\leqslant\mu
\end{equation}
Then $\forall t\in I$,
\[\infnorm{z(t)-y(t)}\leqslant\mu\]
where $k=\degp{p}$.
\end{corollary}

\section{Taylor approximation}\label{sec:taylor_approx}

We first recall a simplified form of Taylor-Lagrange theorem which will be useful for our approximation step of the solution.
\begin{proposition}[Taylor-Lagrange]\label{prop:taylor_lagrange}
Let $a,x\in\R$, $f\in C^{k+1}([a,x])$, then
\[\left|f(x)-\taylor{a}{k}{f}(x)\right|\leqslant\frac{(x-a)^{k}}{k!}\fnsup{a}{f^{(k)}}(x)\]
\end{proposition}

The idea is now to apply this result to the solution of \eqref{eq:ode} and extend it in two directions:
\begin{itemize}
\item Since the solution satisfies \eqref{eq:ode}, we can use \corref{cor:nth_derivative_simpl} to estimate the high-order derivative and the error bound.
\item Since our algorithm will make slight errors, we do not assume that we have the right initial condition; we want a general result with a perturbed solution and relate it to the expected solution using \corref{cor:dependency_init_cond_simp}.
\end{itemize}

\begin{proposition}\label{prop:taylor_div_error}
Let $I=[a,b]$ and $y_0,z_0\in\R^d$. Assume that $y=\Phi(a,y_0,\cdot)$ and $z=\Phi(a,z_0,\cdot)$ are defined over $I$. Let
$Y=\fnsup{a}{y}$ and $\mu\leqslant\frac{1}{3}$ and
assume that $\forall t\in I$,
\begin{equation}
\infnorm{y_0-z_0}\exp\left(k4^k\sigmap{p}|t-a|(1+Y(t)^{k-1})\right)\leqslant\mu
\end{equation}
Then $\forall t\in I$,
\[\infnorm{y(t)-\taylor{a}{n}{z}(t)}\leqslant\mu+\left(d(t-a)\left(1+k!\sigmap{p}(1+\mu+Y(t))^k\right)\right)^n\]
where $k=\degp{p}$.
\end{proposition}

\begin{proof}
Let $\Delta=\infnorm{y(t)-\taylor{a}{n}{z}(t)}$, then
\begin{align*}
\Delta
    &\leqslant\infnorm{y(t)-z(t)}+\infnorm{z(t)-\taylor{a}{n}{z}(t)}\\
    \intertext{Apply \corref{cor:dependency_init_cond_simp} and \propref{prop:taylor_lagrange}}
    &\leqslant\mu+\frac{(t-a)^n}{n!}\fnsup{a}{z^{(n)}}(t)\\
    \intertext{Apply \corref{cor:nth_derivative_simpl}}
    &\leqslant\mu+\frac{(t-a)^n}{n!}\multifac{n}{d}\max\big(1,k!\sigmap{p}\max(1,\fnsup{a}{z}(t)^k)\big)^n\\
    \intertext{Use $\infnorm{z(t)}\leqslant\infnorm{y(t)}+\mu$ and apply \corref{cor:dependency_init_cond_simp}}
    &\leqslant\mu+\frac{(t-a)^n}{n!}\multifac{n}{d}\max\big(1,k!\sigmap{p}\max(1,(\mu+Y(t))^k)\big)^n\\
    &\leqslant\mu+\frac{(t-a)^n}{n!}\multifac{n}{d}\max\big(1,k!\sigmap{p}(1+\mu+Y(t))^k\big)^n\\
    &\leqslant\mu+\left(d(t-a)\left(1+k!\sigmap{p}(1+\mu+Y(t))^k\right)\right)^n\\
\end{align*}
\end{proof}

\section{Our main result}\label{sec:result}

First we need a lemma that will be helpful to compute the forward error.

\begin{lemma}\label{lem:rec_seq_geom_arith}
Let $a>1$ and $b\geqslant0$, assume $u\in\R^\N$ satisfies:
\[u_{n+1}\leqslant au_n+b,\quad n\geqslant0\]
Then
\[u_n\leqslant a^nu_0+b\frac{a^n-1}{a-1},\quad n\geqslant0\]
\end{lemma}

\begin{proof}
By induction, the case $n=0$ is trivial and the induction step works as follows:
\begin{align*}
u_{n+1}&\leqslant au_n+b\\
    &\leqslant a^{n+1}u_0+ab\frac{a^n-1}{a-1}+b\\
    &\leqslant a^{n+1}u_0+b\frac{a(a^n-1)+(a-1)}{a-1}
\end{align*}
\end{proof}

\begin{algorithm}
\DontPrintSemicolon
\SetKwInOut{Input}{input}
\SetKwInOut{Output}{output}
\SetKwFunction{NthDeriv}{NthDeriv}
\Input{The polynomial $p$ of the PIVP}
\Input{The value $z\in\Q^d$ of the function}
\Input{The order $n$ of the derivative}
\Input{The precision $\xi$ requested}
\Output{$x\in\Q^d$}
\emph{Compute $x$ such that $\infnorm{x-y^{(n)}(0)}\leqslant e^{-\xi}$ where $y=\Phi(0,z,\cdot)$ using \propref{prop:nth_derivative}}
\caption{NthDeriv}\label{alg:nth_deriv}
\end{algorithm}

\begin{algorithm}
\DontPrintSemicolon
\SetKwInOut{Input}{input}
\SetKwInOut{Output}{output}
\SetKwFunction{NthDeriv}{NthDeriv}
\Input{The initial condition $(t_0,y_0)\in\Q\times\Q^d$}
\Input{The polynomial $p$ of the PIVP}
\Input{The total time step $T\in\Q$}
\Input{The precision $\xi$ requested}
\Input{The number of steps $N$}
\Input{The order of the method $\omega$}
\Output{$x\in\Q^d$}
\Begin{
    $\Delta\leftarrow\frac{T}{N}$\;
    $x\leftarrow y_0$\;
    \For{$n\leftarrow1$ \KwTo $N$}{
        $x\leftarrow\sum_{i=0}^{\omega-1}\frac{\Delta^i}{i!}$\NthDeriv{$p,t_0+n\Delta,x,\omega,\xi+\Delta$}\;
    }
}

\caption{SolvePIVP}\label{alg:solve_pivp}
\end{algorithm}

Since at each step of the algorithm we compute an approximation of the derivatives, we need a technical result to compute the total error made. We want to relate the error between the value computed by the algorithm and the Taylor approximation, to the error made by computing the derivatives.

\begin{lemma}\label{lem:algo_approx_correct}
Let $n,\xi,d\in\N^*$, $\Delta\in\Q_+$, $z,\tilde{z}\in(\R^d)^n$, assume that $\infnorm{z_i-\tilde{z}_i}\leqslant e^{-\xi-\Delta}$, then
\[\infnorm{\sum_{k=0}^{n-1}\frac{\Delta^k}{k!}z_k-\sum_{k=0}^{n-1}\frac{\Delta^k}{k!}\tilde{z}_k}\leqslant e^{-\xi}\]
\end{lemma}

\begin{proof}
\begin{align*}
\infnorm{\sum_{k=0}^{n-1}\frac{\Delta^k}{k!}(z_k-\tilde{z}_k)}
  &\leqslant e^{-\xi-\Delta}\sum_{k=0}^{n-1}\frac{\Delta^k}{k!}\\
  &\leqslant e^{-\xi-\Delta}e^{\Delta}
\end{align*}
\end{proof}

We now get our main result in a technical form:

\begin{theorem}\label{th:pivp_algo_solve}
If $y$ satisfies \eqref{eq:ode} for $t\in I=[t_0,t_0+T]$, let $k=\degp{p}$, $\mu\in\N$, $T\in\Q_{+},Y\in\Q$ such that
\[\mu\geqslant2\qquad Y\geqslant\fnsup{t_0}{y}(t_0+T)\]
Then \algref{alg:solve_pivp} above guarantees
\[\infnorm{y(t_0+T)-\operatorname{SolvePIVP}(t_0,\tilde{y}_0,p,T,\omega,N,\omega)}\leqslant e^{-\mu}\]
with the following parameters
\[M=(2+Y)^k\quad\; A=d(1+k!\sigmap{p}M)\quad\; N=\lceil TeA\rceil\]
\[\Delta=\frac{T}{N}\quad\; B=k4^k\sigmap{p}\Delta M\quad\;\omega=2+\mu+\ln(N)+NB\]
\[\infnorm{y_0-\tilde{y}_0}\leqslant e^{-NB-\mu-1}\]
\end{theorem}

\begin{proof}
Denote by $t_n=t_0+n\Delta$ and $x_n$ the value of $x$ at the $n^\text{th}$ step of the algorithm. That is:
\[y^{[n]}=\Phi(t_n,x_n,\cdot)\qquad \infnorm{x_{n+1}-\taylor{t_n}{\omega}{y^{[n]}}(t_n+\Delta)}\leqslant e^{-\omega}\]
Notice that $x_{n+1}$ is only an approximation of the Taylor approximation. We request an approximation up to $e^{-\omega}$, as a parameter of the algorithm. \lemref{lem:algo_approx_correct} ensures that this bound is indeed reached by computing the derivatives sufficiently precisely. Then define
\[\varepsilon_n=\infnorm{x_n-y(t_n)}\]
By the choice of $x_0=\tilde{y}_0$, we have $\varepsilon_0=\infnorm{\tilde{y}_0-y_0}\leqslant e^{-NB-\mu-1}$.
Now assume that $\varepsilon_ne^B\leqslant\frac{1}{3}$. After a few simplifications, \propref{prop:taylor_div_error} gives:
\[\varepsilon_{n+1}\leqslant\varepsilon_ne^B+(\Delta A)^\omega+e^{-\omega}\]
Now apply \lemref{lem:rec_seq_geom_arith}:
\[\varepsilon_n\leqslant e^{nB}\varepsilon_0+\left((\Delta A)^\omega+e^{-\omega}\right)\frac{e^{nB}-1}{e^B-1}\]
Notice that
\[\frac{e^{nB}-1}{e^B-1}=\sum_{k=0}^{n-1}e^{kB}\leqslant ne^{nB}\]
Thus for $n=N$ we get
\[\varepsilon_N\leqslant e^{NB}e^{-NB-\mu-1}+\left((\Delta A)^\omega+e^{-\omega}\right) Ne^{NB}\]
By the choice of $A$ we have
\[\Delta A\leqslant e^{-1}\]
Thus by the choice of $\omega$ we have
\begin{align*}
\varepsilon_N
  &\leqslant e^{-\mu-1}+2e^{-2-\mu-\ln(N)-NB}Ne^{NB}\\
  &\leqslant e^{-\mu-1}+e^{-\mu-1}\leqslant e^{-\mu}\leqslant\frac{1}{3}
\end{align*}
Notice that we need to check that we indeed get a value smaller that $\frac{1}{3}$ at the end otherwise we couldn't have applied \propref{prop:taylor_div_error}.
\end{proof}


\begin{lemma}\label{lem:nth_deriv_polytime}
If the coefficients of the vector of polynomial $p$ are polynomial time computable, then for all $z\in\Q^d$ and $n,\xi\in\N^*$, $\operatorname{NthDeriv}(p,z,n,\xi)$ has running time polynomial in the value of $n$ and $\xi$.
\end{lemma}

\begin{proof}
From \propref{prop:nth_derivative}, we have:
\[\forall t\in I,\forall n\in\N^{*}, y_i^{(n)}(t)=Q_{i,n}(V(t))\]
where $Q_{i,n}$ is a polynomial of degree at most $n$ and $\sigmap{Q_{i,n}}\leqslant\multifac{d}{(n-1)}$. From the proof, it is easy to see that the $Q_{i,n}$ are computable by induction in polynomial time in $n$ and $d$ since $Q_{i,n}$ is of degree $n$ (thus has at most $n^d$ terms) and has the sum of its coefficients not larger that $\multifac{d}{(n-1)}$ (thus taking a space and time at most polylogarithmic in $n$ to manipulate). Finally, $V(t)$ can be computed with precision $e^{-\xi}$ in time polynomial in $\xi$ from the assumption on the coefficients of $p$.
\end{proof}

\begin{corollary}
If the coefficients of the vector of polynomial $p$ are polynomial time computable, then for all $t_0\in\Q, y_0\in\Q^d, T\in\Q, \xi,N,\omega\in\N$, $\operatorname{SolvePIVP}(t_0,y_0,p,T,\xi,N,\omega)$ has running time polynomial in the value of $T/N,\xi,N$ and $\omega$.
\end{corollary}

In less technical form:

\begin{theorem}\label{th:pivp_algo_solve_simp}
There exists an algorithm $\mathcal{A}$ such that for any $p$ vector of polynomial with polynomial time computable coefficients, $y_0\in\R^d$ polynomial time computable vector, $t_0\in\Q$, $\mu\in\N$, $T\in\Q$ and $Y\in\Q$ such that $Y\geqslant\fnsup{t_0}{y}(t_0+T)$,
\[\infnorm{\mathcal{A}(p,y_0,t_0,\mu,T,Y)-\Phi(t_0,y_0,t_0+T)}\leqslant e^{-\mu}\]
Furthermore $\mathcal{A}(p,y_0,t_{0},\mu,T,Y)$ is computed in time polynomial in the value of $\mu, T$ and $Y$.
\end{theorem}

\section{Extension}

Our main result has the nice property that it requires minimal hypothesis and has maximal precision in its statement: all the parameters of the polynomials are kept. On the other hand, only considering differential equations of the form $\dot{y}=p(y)$ has two drawbacks:
\begin{itemize}
\item it is not always possible;
\item even when possible, it might considerably increase the size of the system, and thus the complexity (which can be exponential is the size of the system) compared to a "non-expanded form" (e.g $\dot{y}=\fiter{\sin}{k}(y)$ will expand to a $2k$-dimensional system).
\end{itemize}
For this reason, we want to extend our result to differential equations of the form $\dot{y}=f(y)$ with restrictions on $f$. Intuitively, our hypothesis will be the lemmas we had for $p$ in the previous sections. That is, $f$ will need to be computable quickly and its derivatives must not grow too fast. We only give the lemmas and proposition which are similar to the previous sections. The proofs are similar and we make a few remarks when relevant. For simplicity, we write the hypothesis only once, after introducing a small notation. Notice that in this section we assume $d$ is a constant (since everything will be exponential in $d$ anyway, we do not need to consider it). \hypref{hyp:ext_f} is about the differential equation verified by $y$. \hypref{hyp:ext_f_deriv} is about value the derivatives of $f$. \hypref{hyp:ext_f_lip} is about the continuity modulus of $f$ (or its "Lipschitz" constant). Finally, \hypref{hyp:ext_f_comp} is about the complexity of computing $f$ and its derivatives. 
\medskip

\newcommand{\poly}[1]{\operatorname{poly}(#1)}
\newcommand{\polypoly}[2]{{\operatorname{poly}^*}\left(#1\thinspace\#\thinspace #2\right)}
\noindent\textbf{Notation:} $\polypoly{X}{Y}=\poly{X}^{\poly{Y}}$

\begin{hypothesis}\label{hyp:ext_f}
Let $I\subseteq\R, J\subseteq\R^d$, $y:I\rightarrow J$, $f:J\rightarrow\R^d$, $t_0\in\Q$, $y_0\in\R^d$. Assume that
\[\left\{\begin{array}{@{}c@{}l}\dot{y}&=f(y)\\y(t_0)&=y_0\end{array}\right.\]
\end{hypothesis}
\begin{hypothesis}\label{hyp:ext_f_deriv}
Assume that for all $\alpha\in\N^d$, $x\in J$,
\[\infnorm{f^{(\alpha)}(x)}\leqslant\polypoly{\infnorm{x},|\alpha|}{|\alpha|}\]
\end{hypothesis}

\begin{hypothesis}\label{hyp:ext_f_lip}
Assume that for all $a,b\in J$,
\[\infnorm{f(a)-f(b)}\leqslant\infnorm{a-b}Q(\max\infnorm{a},\infnorm{b})\]
where $Q$ is a polynomial of degree $k$.
\end{hypothesis}

\begin{hypothesis}\label{hyp:ext_f_comp}
Assume that $y_0$ is polynomial time computable vector of reals and that $f$ and its derivative are polynomial computable, that is, for all $\alpha\in\N$, $x\in J$, $f^{(\alpha)}(x)$ is computable with precision $e^{-\xi}$ in time polynomial in the value of $|\alpha|$, $\xi$ and $\infnorm{x}$.
\end{hypothesis}

\hypref{hyp:ext_f} will provides us a with a recursive formula for the derivatives of $y$, similarly to \propref{prop:nth_derivative}. Armed with \hypref{hyp:ext_f_deriv} which more or less replaces \lemref{lem:poly_deriv_major_sigmap_deg}, we can derive a bound on the derivative of $y$, similarly to \corref{cor:nth_derivative_simpl}.

\begin{proposition}\label{prop:ext_nth_derivative}
Define
\[\Gamma_n=\big\{\alpha\in\N^d\thinspace\big|\thinspace|\alpha|\leqslant n\big\}\qquad\Lambda_n=\llbracket1,d\rrbracket\times\Gamma_n\]
\[V_n(t)=\left(f_i^{(\alpha)}(y(t))\right)_{(i,\alpha)\in\Lambda_n}\qquad t\in I\]
Then
\[\forall t\in I,\forall n\in\N^{*}, y_i^{(n)}(t)=Q_{i,n}(V_{n-1}(t))\]
where $Q_{i,n}$ is a polynomial of degree at most $n$. Furthermore,
\[\sigmap{Q_{i,n}}\leqslant\multifac{d}{(n-1)}\]
\end{proposition}

\begin{corollary}\label{cor:ext_nth_derivative_simpl}
\[\infnorm{y^{(n)}(t)}\leqslant\polypoly{\infnorm{y(t)},n}{n}\]
\end{corollary}

Similarly to \propref{prop:dependency_init_cond}, \hypref{hyp:ext_f_lip} will allow us to bound the divergence of two solutions given the initial difference. We reuse the notation $y=\Phi(a,y_0,t)$ for $f$ with its obvious meaning.

\begin{proposition}\label{prop:ext_dependency_init_cond}
Let $K=[a,b]$ and $y_0,z_0\in\R^d$. Assume that $y=\Phi(a,y_0,\cdot)$ and $z=\Phi(a,z_0,\cdot)$ are defined over $K$. Let $Y=\fnsup{a}{y}$.
Assume that $\forall t\in I$,
\[\infnorm{y_0-z_0}\sigmap{Q}(1+Y(t))^k\exp\left(\sigmap{Q}|t-a|(1+Y(t))^k\right)\leqslant\frac{1}{3}\]
Then $\forall t\in I$,
\[\infnorm{z(t)-y(t)}\leqslant4\infnorm{y_0-z_0}e^{\sigmap{Q}|t-a|(1+Y(t))^k}\]
\end{proposition}

The other results are basically the same except for the exact constant choices in the theorem. We get the same result at the end, that is $y(t)$ is computable is polynomial time with respect to the same parameters.

\section{Acknowledgments} D.S. Gra\c{c}a was partially supported by
\emph{Funda\c{c}\~{a}o para a Ci\^{e}ncia e a Tecnologia} and EU FEDER
POCTI/POCI via SQIG - Instituto de Telecomunica\c{c}\~{o}es through
the FCT project PEst-OE/EEI/LA0008/2011.

\bibliography{ContComp}

\begin{thebibliography}{BCGH07b}

\bibitem[BC08]{CIEChapter2007}
Olivier Bournez and Manuel~L. Campagnolo.
\newblock {\em New Computational Paradigms. Changing Conceptions of What is
  Computable}, chapter A Survey on Continuous Time Computations, pages
  383--423.
\newblock Springer-Verlag, New York, 2008.

\bibitem[BCGH07a]{BCGH07}
O.~Bournez, M.~L. Campagnolo, D.~S. Gra{\c{c}}a, and E.~Hainry.
\newblock Polynomial differential equations compute all real computable
  functions on computable compact intervals.
\newblock {\em J. Complexity}, 23(3):317--335, 2007.

\bibitem[BCGH07b]{JOC2007}
Olivier Bournez, Manuel~L. Campagnolo, Daniel~S. Gra{\c c}a, and Emmanuel
  Hainry.
\newblock Polynomial differential equations compute all real computable
  functions on computable compact intervals.
\newblock {\em Journal of Complexity}, 23(3):317--335, June 2007.

\bibitem[Bus31]{Bus31}
V.~Bush.
\newblock The differential analyzer. {A} new machine for solving differential
  equations.
\newblock {\em J. Franklin Inst.}, 212:447--488, 1931.

\bibitem[Dem96]{Dem96}
J.-P. Demailly.
\newblock {\em Analyse Num\'{e}rique et \'{E}quations Diff\'{e}rentielles}.
\newblock Presses Universitaires de Grenoble, 1996.

\bibitem[GC03]{GC03}
D.~S. Gra{\c{c}}a and J.~F. Costa.
\newblock Analog computers and recursive functions over the reals.
\newblock {\em J. Complexity}, 19(5):644--664, 2003.

\bibitem[GCB08]{GCB08}
D.~S. Gra{\c{c}}a, M.~L. Campagnolo, and J.~Buescu.
\newblock Computability with polynomial differential equations.
\newblock {\em Adv. Appl. Math.}, 40(3):330--349, 2008.

\bibitem[Sha41]{Sha41}
C.~E. Shannon.
\newblock Mathematical theory of the differential analyzer.
\newblock {\em J. Math. Phys. MIT}, 20:337--354, 1941.

\bibitem[Smi06]{Smi1}
Warren~D. Smith.
\newblock Church's thesis meets the {N}-body problem.
\newblock {\em Applied Mathematics and Computation}, 178(1):154--183, 2006.

\bibitem[Wei00]{Wei00}
K.~Weihrauch.
\newblock {\em Computable Analysis: an Introduction}.
\newblock Springer, 2000.

\bibitem[Wil96]{IEEEAnnals}
Michael~R. Williams.
\newblock About this issue.
\newblock {\em IEEE Annals of the History of Computing}, 18(4),
  October--December 1996.

\end{thebibliography}
\bibliographystyle{alpha}

\end{document}